\documentclass[11pt, oneside]{amsart}
\usepackage{amsfonts, amstext, amsmath, amsthm, amscd, amssymb}
\usepackage{commath}

\usepackage{url}

\usepackage[paper=a4paper, text={138mm,208mm},centering]{geometry}
\usepackage[bookmarks]{hyperref}
\usepackage{float, placeins, booktabs, longtable}
\usepackage[capitalise]{cleveref}
\usepackage{enumitem}
\usepackage{tikz}
\usetikzlibrary{intersections,patterns,decorations.markings}
\usepackage{tikz-cd}
\crefname{conditionsi}{condition}{conditions}
\Crefname{conditionsi}{Condition}{Conditions}

\newcommand{\pfstep}[3]{
  \smallskip
  \hypertarget{#3}{\textbf{Step #1. #2.}}
}

\numberwithin{equation}{section}

\theoremstyle{plain}
\newtheorem{theorem}[equation]{Theorem}
\newtheorem{lemma}[equation]{Lemma}
\newtheorem{proposition}[equation]{Proposition}

\newtheorem{corollary}[equation]{Corollary}

\theoremstyle{definition}

\newtheorem{remark}[equation]{Remark}

\newtheorem{definition}[equation]{Definition}
\newtheorem{cond}[equation]{Condition}

\newtheorem*{ack}{Acknowledgments}

\DeclareMathOperator\Int{Int}
\DeclareMathOperator\ind{ind}

\newcommand{\R}{\mathbb R}
\newcommand{\xip}{\xi'}
\newcommand{\xil}{\xi_{\mathrm{lin}}}
\newcommand{\Xin}{X_{\mathrm{in}}}
\newcommand{\Xout}{X_{\mathrm{out}}}
\newcommand{\Xtan}{X_{\mathrm{tan}}}
\renewcommand{\setminus}{{\smallsetminus}}
\newcommand{\wt}[1]{\widetilde{#1}}
\newcommand{\ol}[1]{\overline{#1}}
\newcommand{\WW}{W}
\newcommand{\WWII}{\WW_{\mathrm{II}}}

\newcommand{\br}[1]{\left\lbrace #1 \right\rbrace}
\newcommand\restr[2]{{
  \left.\kern-\nulldelimiterspace 
  #1 
  \vphantom{\big|} 
  \right|_{#2} 
}}

\begin{document}

\title{Merging boundary critical points of a Morse function}

\author{Maciej Borodzik}
\address{Institute of Mathematics, Polish Academy of Sciences, ul. Sniadeckich 8, 00-656 Warsaw, Poland}
\email{mcboro@mimuw.edu.pl}

\author{Marcin Mielniczuk}
\address{Institute of Informatics, University of Warsaw, ul. Banacha 2,
  02-097 Warsaw, Poland}
\email{m.mielniczuk@uw.edu.pl}

\def\subjclassname{\textup{2010} Mathematics Subject Classification}
\expandafter\let\csname subjclassname@1991\endcsname=\subjclassname
\expandafter\let\csname subjclassname@2000\endcsname=\subjclassname

\begin{abstract}
  In 2015, Borodzik, N\'emethi and Ranicki
  proved that an interior critical point can be pushed to the boundary, where it splits into two boundary critical points.
  In this paper, we show that two critical points at the boundary can be, under specific assumptions, merged into
  a single critical point in the interior. That is, we reverse the original construction.
\end{abstract}
\maketitle

\section{Introduction}
While the origins of the study of Morse functions for manifolds with boundary are in the seventies of the previous century,
a systematic research of this subject was originated only in the first decade of the twenty-first century
by Kronheimer and Mrowka \cite{KronheimerMrowka}.
Their analysis of the notions of boundary stable and unstable critical points and the Morse-Witten chain complex for
manifolds with boundary is put in the context of Floer theory.

In 2015, N\'emethi, Ranicki and the first author proved that an interior critical point of a Morse function can be pushed to the
boundary, where it splits into two boundary critical points; see \cite{morse-wboundary}.
A precise statement of the result is the following (we refer to
Section~\ref{sec:morse_intro} for terminology):
\begin{theorem}[see \expandafter{\cite[Theorem 3.1]{morse-wboundary}}]\label{thm:bnr}
  Suppose $M$ is a smooth manifold with boundary and $g\colon M\to\R$ is a Morse function.
  Suppose that $z\notin\partial M$ is a critical point of $g$ of index $k\neq 0,\dim M$. Suppose there is a path $\theta$ in the level
  set $g^{-1}(g(z))$ connecting $z$ and $\partial M$.

  In this situation, there exists a Morse function $f$ such that:
  \begin{enumerate}
    \item $f(x)=g(x)$ away from a neighborhood $U$ of $\theta$, where $U$ might be chosen to be as small as we please;
    \item $f(x)$ has precisely two critical points $p$ and $q$ in $U$. Both $p$ and $q$ belong to $\partial M$
          and have index $k$.
          The critical point $p$ is boundary stable, while the critical point $q$
          is boundary unstable;\label{thm:bnr:2crits}
    \item there is a gradient-like Morse--Smale vector field for $f$, such that there
          is precisely one trajectory $\gamma$ connecting $p$ to $q$.\label{thm:bnr:single-traj}
  \end{enumerate}
\end{theorem}

In the present paper, we prove the reverse result. Namely, under suitable circumstances, a pair of critical points on the boundary
can be merged into a single critical point in the interior. The precise statement follows. We refer to Figure~\ref{fig:main_thm}
for explanation.
\begin{theorem}\label{thm:main}
  Let $M$ be a manifold with boundary, $f\colon M\to\R$ be a Morse function,
  and let $\xi$ be a Morse--Smale gradient-like vector field for $f$.
  Suppose that $p, q \in \partial M$ are critical points of $f$, both of index $k$,
  where $p$ is boundary stable and $q$ is boundary unstable.
  Moreover, suppose there exists a single trajectory $\gamma$ of the vector field $\xi$
  starting at $p$ and terminating at $q$.

  Then for any neighborhood $U$ of $\gamma$, there exists a Morse function $g\colon M\to\R$,
  with the same critical points as $f$ away from $U$ and a single critical point of index $k$ inside $U$.
\end{theorem}
\begin{figure}
  \begin{tikzpicture}[decoration={markings,mark=at position 0.5 with {\arrow{>}}}]
    \begin{scope}[xshift=-3.5cm]
      \fill[yellow!05] (0,1.5) rectangle (2,-1.5);
      \draw[thick,postaction={decorate}] (0,-0.6) -- (0,0.6);
      \fill[black] (0,-0.6) circle (0.08) node [left] {$p$};
      \fill[black] (0,0.6) circle (0.08) node [left] {$q$};
      \draw[thick,postaction={decorate}] (0,-0.6) -- (0,-1.5);
      \draw[thick,postaction={decorate}] (0,1.5) -- (0,0.6);
      \draw[postaction={decorate}] (0,0.6) .. controls ++ (0.3,0) and ++ (0,-0.3) .. (1.5,1.5);
      \draw[postaction={decorate}] (1.5,-1.5) .. controls ++ (0,0.3) and ++ (0.3,0) .. (0,-0.6);
      \draw[postaction={decorate}] (1.25,-1.5) .. controls ++ (-0.4,0.6) and ++ (0.05,0.8) .. (0.25,-1.5);
      \draw[postaction={decorate}] (0.25,1.5) .. controls ++ (0.05,-0.8) and ++(-0.4,-0.6) .. (1.25,1.5);
      \draw[postaction={decorate}] (1.75,-1.5) .. controls ++ (0,0.5) and ++ (0,-0.6) .. (0.2,0);
      \draw[postaction={decorate}] (0.2,0) .. controls ++(0,0.6) and ++(0,-0.5) .. (1.75,1.5);
    \end{scope}
    \begin{scope}[xshift=3cm]
      \fill[yellow!05] (0,1.5) rectangle (2,-1.5);
      \draw[thick,postaction={decorate}] (0,1.5) -- (0,-1.5);
      \draw[postaction={decorate}] (0.5,1.5) -- (1,0);
      \draw[postaction={decorate}] (1.5,-1.5) -- (1,0);
      \draw[postaction={decorate}] (1,0) -- (0.5,-1.5);
      \draw[postaction={decorate}] (1,0) -- (1.5,1.5);
      \draw[postaction={decorate}] (0.25,1.5) .. controls ++ (0.2,-0.9) and ++(0.2,0.9) .. (0.25,-1.5);
      \draw[postaction={decorate}] (1.75,-1.5) .. controls ++ (-0.2,0.9) and ++(-0.2,-0.9) .. (1.75,1.5);
      \draw[postaction={decorate}] (0.75,1.5) .. controls ++ (0.1,-0.9) and ++(-0.1,-0.9) .. (1.25,1.5);
      \draw[postaction={decorate}] (1.25,-1.5) .. controls ++ (-0.1,0.9) and ++(0.1,0.9) .. (0.75,-1.5);
      \fill[black] (1,0) circle (0.08) node [right] {$z$};
    \end{scope}
    \draw[dashed,->] (-1.3,0.3) .. controls ++(0.5,0.5) and ++(-0.5,0.5) .. node [above] {\cref{thm:main}} (2.5,0.3);
    \draw[dotted,->] (2.5,-0.3) .. controls ++(-0.5,-0.5) and ++(0.5,-0.5) .. node [below] {\cref{thm:bnr}} (-1.3,-0.3);

  \end{tikzpicture}
  \caption{Theorems~\ref{thm:bnr} and~\ref{thm:main} in dimension 2.}\label{fig:main_thm}
\end{figure}

Theorem~\ref{thm:main} deals with the only remaining case of possible cancellations of boundary critical points as explained
in the following corollary.
\begin{corollary}
  Suppose $p$ and $q$ are boundary critical points of $f$ such that the indices of $p$ and $q$ computed on $\partial M$
  differ by $1$ (with $f(p)<f(q)$ and $\ind p < \ind q$). Suppose $\xi$ is a gradient-like Morse--Smale vector field for $f$,
  and such that there exists a single trajectory of $\xi$ from $p$ to $q$. Then:
  \begin{itemize}
    \item if $p$ and $q$ are both boundary stable (resp. boundary unstable), then they can be canceled (as in \cite[Theorem 5.1]{morse-wboundary});
    \item if $p$ is boundary stable and $q$ is boundary unstable, then the pair $p$ and $q$ can be merged and pushed into
          the interior as in Theorem~\ref{thm:main};
    \item if $p$ is boundary unstable and $q$ is boundary stable, then $\ind q=\ind p+2$, see the table in \cite[Section 4.3]{morse-wboundary}.
          This means that $p$ and $q$ are usually connected by a 1-parameter family of trajectories. No simplification is possible.
  \end{itemize}
\end{corollary}
We remark that in the last case on the bullet list above,
the non-cancellation theorem \cite[Lemma 5.7]{morse-wboundary} can be applied as well.

The aim of the paper is to prove Theorem~\ref{thm:main}. The idea is
to construct a suitable coordinate system in $U$, which might require replacing $\xi$
with another gradient-like vector field $\wt{\xi}$ for $f$. Then, we alter $\wt{\xi}$ to
obtain another vector field $\xip$
on $M$ and show that we can construct a Morse function $g$ for which $\xip$ is gradient-like.

The actual plan of the paper is as follows.
In Section~\ref{sec:morse_intro}, we introduce the necessary terminology
and recall the relevant results from \cite{morse-wboundary}.
In Section~\ref{sec:coordinate}, we construct the suitable coordinate system. Finally, in Section~\ref{sec:main_proof},
we give the proof of Theorem~\ref{thm:main}.

\begin{ack}
  The paper is a part of the Master's thesis of MM under the supervision of MB.
  A part of the project was done while MB was visiting the R\'enyi Institute, whose hospitality he is grateful for.
  Both authors were supported by the NCN OPUS Grant 2019/B/35/ST1/01120.
\end{ack}
\section{Morse theory for manifolds with boundary}\label{sec:morse_intro}

Recall that a smooth function $f\colon M\to\R$, where $M$ is a closed smooth manifold, is called \emph{Morse}, if any
critical point $x$ of $f$ is non-degenerate, that is to say, the matrix of the second derivatives $D^2f(x)$ is non-degenerate,
which turns out not to depend on the choice of the coordinate system around $x$.

This definition can be generalized to the case of manifolds with boundary; see \cite[Section 2.4]{KronheimerMrowka} and \cite[Section 1]{morse-wboundary}.
\begin{definition}[Morse function]\label{def:morsefn}
  Let $M$ be a manifold with boundary, and let $f\colon M \to \R$ be smooth.
  We say that $f$ is \emph{Morse} if:
  \begin{enumerate}
    \item all critical points of $f$ are non-degenerate;
    \item $f|_{\partial M}$ is Morse in the usual sense;
    \item all critical points of $f|_{\partial M}$ are actually critical points of $f$. \label{no-bad-crits}
  \end{enumerate}
\end{definition}

This definition differs from the ones given in \cite{morse-wboundary,KronheimerMrowka} and \cite{BorodzikPowell}.
However, condition \ref{no-bad-crits} of \cref{def:morsefn} is equivalent to saying that $T_p\partial M \subset \ker df$ implies $dF = 0$,
which is the contrapositive of condition (2.4) of \cite{BorodzikPowell}, and we can repeat their argument.

\begin{theorem}[Boundary Morse lemma, see \expandafter{\cite[Lemma 2.6]{morse-wboundary}}]\label{lem:boundary_morse}
  Let $p \in \partial M$ be a non-degenerate critical point of a Morse function $f\colon M \to \R$.
  Then there exist an integer $k=0,\dots,n-1$, $\epsilon=\pm1$, and local coordinates $y,x_1, \dots, x_{n-1}$, defined in an open neighborhood $U \ni p$, such that
  \begin{enumerate}
    \item $p=(0,\dots,0)$;
    \item $y\ge 0$ on $U$;
    \item $y=0$ defines $\partial M\cap U$;
    \item the following equality holds:
          \begin{equation}\label{eq:morse}
            f(x_1,\dots,x_{n-1},y) = f(p) - x_1^2 - \dots - x_k^2 + x_{k+1}^2 + \dots + x_{n-1}^2+\epsilon y^2.
          \end{equation}
  \end{enumerate}
\end{theorem}
The next definition comes from \cite{KronheimerMrowka}, we refer to \cite[Section 2.4]{morse-wboundary} for a detailed
discussion.
\begin{definition}[Boundary stable and unstable critical points]
  Let $p$ be a boundary critical point of a Morse function $f$. We say that $p$ is \emph{boundary stable} (resp.  \emph{boundary unstable}) if
  $\epsilon=-1$ (resp. $\epsilon=+1$), where $\epsilon$ is defined as in \eqref{eq:morse}.
\end{definition}
In short, being a boundary stable critical point means that
the flow of $\nabla f$ attracts toward the boundary in the vicinity of the critical point,
while in the case of a boundary unstable critical point,
the flow of $f$ repels from the boundary.

We now define the index of a boundary critical point.
\begin{definition}[Index of a critical point]
  Let $p$ be a boundary critical point. The \emph{index} of a critical point is the dimension of the negative definite subspace
  of $D^2f(p)$. Put differently, for a boundary stable critical point, the index is $k+1$, while for the boundary unstable critical
  point, the index is $k$. Here, $k$ is defined as in \cref{lem:boundary_morse}.
\end{definition}

We now recall the definition of a gradient like vector field.
\begin{definition}[Gradient-like vector field]\label{def:gl-vf}
  Let $f$ be a Morse function on a manifold with boundary $M$.
  Let $\xi$ be a vector field on $M$. We shall say that $\xi$ is gradient-like with respect to $f$, if
  \begin{enumerate}
    \item $\xi$ is tangent to $\partial M$ at the boundary;
    \item $\partial_\xi f > 0$ away from the critical points of $f$;
    \item for any critical point $p$ of $f$, there exist local coordinates $x_1, \dots, x_n$ around $p$, such that $f$ and $\xi$ admit
          the following form (called the Morse normal form):
          \begin{equation}\label{eq:morse_normal}
            \begin{split}
              f(\bar x)   =& f(p) - x_1^2 - \dots - x_k^2 + x_{k+1}^2 + \dots + x_n^2 \\
              \xi(\bar x) =& (-x_1, \dots, -x_k, x_{k+1}, \dots, x_n)
            \end{split}
          \end{equation}
  \end{enumerate}
\end{definition}
The model situation for $\xi$ is that $\xi=\nabla f$, for a suitably chosen Riemannian metric on $M$.
In \cite[Section 1.1]{morse-wboundary}, it is proved that
each Morse function $f$ admits a gradient-like vector field. We now recall the Morse--Smale condition; see \cite[Section 4.3]{morse-wboundary}.
\begin{definition}[Morse-Smale vector field]
  A gradient-like vector field $\xi$ for $f$ is called \emph{Morse-Smale}, if for any two critical points $p$ and $q$,
  with $f(p)<f(q)$:
  \begin{itemize}
    \item the submanifolds of $\partial M$: $W^u(p)\cap \partial M$ and $W^s(q)\cap\partial M$ intersect transversely,
    \item the submanifolds of $\Int M$: $W^u(p)\cap \Int M$ and $W^s(q)\cap\Int M$ intersect transversely.
  \end{itemize}
  Here $W^s$ and $W^u$ are, respectively, the stable and the unstable manifolds of a critical point of $\xi$.
\end{definition}
We note that the Morse--Smale condition is open-dense among all gradient-like vector fields; see \cite[Section 4.3]{morse-wboundary}.

\section{Coordinate neighborhood}\label{sec:coordinate}

Throughout \cref{sec:coordinate}, we let $f$ be a fixed Morse function
satisfying the assumption of Theorem~\ref{thm:main}. All gradient-like vector fields are
assumed to be gradient-like with respect to the function $f$.

The proof of the following result is completely analogous to the proof of \cite[Proposition 5.2]{morse-wboundary},
with the only difference being that the local behavior of the first coordinate is different.
\begin{proposition}\label{prop:coordinates}
  There exists an open neighborhood $U_1$ of $\gamma$,
  a coordinate map $\varphi\colon U_1\to \R_{\ge 0}\times\R^{n-1}$ (with coordinates denoted by $(y,x_1,\dots,x_{n-1})$)
  and a gradient-like vector field $\xi_1$ for $f$
  agreeing with $\xi$ away from $U_1$, such that:
  \begin{itemize}
    \item $\varphi$ takes $U_1\cap \partial M$ to $\{0\}\times\R^{n-1}$;
    \item $\varphi(p)=(0,0,\dots,0)$;
    \item $\varphi(q)=(0,1,0,\dots,0)$;
    \item the curve $\gamma$ is mapped to the segment $(0,t,0,\dots,0)$, $t\in[0,1]$, connecting $p$ with $q$;
    \item the map $\varphi$ takes $\xi_1$ to a vector field on $\varphi(U_1)\subset \R_{\ge 0}\times\R^{n-1}$
          given in the form
          \begin{equation}\label{eq:form_of_xi_prim}
            (yv(y,x_1,\dots),w(x_1),-x_2,\dots,-x_{k-1},x_{k},\dots,x_{n-1})
          \end{equation}
          for some smooth functions $v, w$ with the properties listed below;
    \item the function $w$ is positive for $x_1 \in (0,1)$ and negative for $x_1 \notin [0,1]$;
    \item the function $v$ is negative at $x_1 = 0$ and positive at $x_1 = 1$.
  \end{itemize}
\end{proposition}
\begin{remark}
  The convention in \eqref{eq:form_of_xi_prim} is that the first coordinate of $\xi_1$ is the $\frac{\partial}{\partial y}$-coordinate,
  while the next coordinates are directions of $\frac{\partial}{\partial x_1},\dots,\frac{\partial}{\partial x_{n-1}}$.
\end{remark}
\begin{remark}
  Just as in the proof of \cite[Proposition 5.2]{morse-wboundary}, the proof of the \lcnamecref{prop:coordinates} above crucially requires
  $\xi$ to be a Morse--Smale vector field.
  This is the only place in the paper where we are using the Morse--Smale assumption.
\end{remark}

From now on, we assume that such $U_1$ and $\varphi$ have been chosen. We will now improve $\xi_1$ so that it still
has the form \eqref{eq:form_of_xi_prim}, but the function
$v$ is better behaved.
\begin{lemma}\label{lem:u2}
  There exists a smaller neighborhood $U_2\subset U_1$ of $\gamma$ and a gradient-like vector field $\xi_2$, agreeing with $\xi_1$
  away from $U_1$, such that, on $U_2$ $D\varphi(\xi_2)$ is given by \eqref{eq:form_of_xi_prim}, but
  \begin{equation}
    v(y,x_1,\dots,x_{n-1})=2x_1-1.
  \end{equation}
\end{lemma}
\begin{proof}
  Define
  \[
    \xi_v=(y(2x_1-1),w(x_1),-x_2,\dots,-x_{k-1},x_{k},\dots,x_{n-1}).
  \]
  Around the critical points, $f$ is given by \eqref{eq:morse}.
  By direct calculation, we obtain that $\partial_{\xi_v}f\ge 0$
  in a neighborhood $U_p$ of $p$ and in a neighborhood $U_q$ of $q$.
  Let $U_{\gamma}$ be a neighborhood of $\gamma$. 

  The vector field $\xi_1$ is gradient-like for $f$ and $\ol{U_\gamma\setminus (U_p\cup U_q)}$ does not contain any critical points of $f$,
  and so there exists a $C>0$ such that $\partial_{\xi_1}f>C$ everywhere on $\ol{U_\gamma\setminus (U_p\cup U_q)}$.
  As $y\equiv 0$ on $\gamma$, the first coordinate of both $\xi_v$ and $\xi_1$ is small at points that are close to the boundary, that is,
  we may assume that $U_{\gamma}$ is small enough that the following two conditions hold on $U_{\gamma}$:
  \begin{align*}
    \abs{yv(y,x_1,\dots,x_{n-1})\frac{\partial f}{\partial y}} & <   C/3; \\
    \abs{y(2x_1-1)\frac{\partial f}{\partial y}}        & <   C/3.
  \end{align*}
  By the triangle inequality, we conclude that $\partial_{\xi_v}f>C/3$ everywhere on $U_{\gamma}\setminus(U_p\cup U_q)$,
  and so $\xi_v$ is gradient-like for $f$ on $U_{\gamma}\cup U_p\cup U_q$.

  To define a global vector field $\xi_2$, choose an open neighborhood $U_2$ of $\gamma$ such that $\ol{U_2}\subset U_{\gamma}\cup U_p\cup U_q$.
  Let $\phi_2\colon M\to[0,1]$
  be a smooth function supported on $U_{\gamma}\cup U_p\cup U_q$ and equal to $1$ on $U_2$. We finally define
  \[\xi_2=\phi_2\xi_v+(1-\phi_2)\xi_{1}.\]
  This vector field clearly has the desired form on $U_2$ and is gradient-like as a convex combination of gradient-like vector fields.
\end{proof}

The following result is a repetition of \cite[Assertion 1]{milnor-hcobordism}. In the applications, we will set $W$
to be an open set containing $\gamma$, whose closure is contained in $U_2$.
\begin{proposition}\label{prop:leaving-crit}
  Suppose $\wt{\xi}$ is a gradient-like vector field for $f$.
  For any open subset $\WW$ containing $\gamma$,
  there exists a neighborhood $U$ of $\gamma$, $U\subset \WW$ such that
  if a trajectory of $\wt{\xi}$ enters $U$ and leaves $\WW$, then it never re-enters $U$.
\end{proposition}

\section{Proof of the main theorem}\label{sec:main_proof}
We begin with the following auxiliary result.
\begin{proposition}\label{prop:no_points_in_between}
  Suppose $\xi$ and $f$ are as in Theorem~\ref{thm:main}. There exists a Morse function $\wt{f}$ having
  the same critical points as $f$, such that $\xi$ is gradient-like for $\wt{f}$ and there is no critical point $q'$
  of $\wt{f}$ other that $p$ and $q$ such that $\wt{f}(q')\in[\wt{f}(p),\wt{f}(q)]$.
\end{proposition}
\begin{proof}
  The result follows from the Global Rearrangement Theorem \cite[Proposition 4.6]{morse-wboundary}.
  The rearrangement can be carried out by successively applying the Elementary Rearrangement Theorem \cite[Proposition 4.1]{morse-wboundary}.
  Note that if one chooses the auxiliary function $\mu$ in the proof of \cite[Lemma 4.3]{morse-wboundary}
  to be preserved by the gradient-like vector field $\xi$ (instead of $\nabla F$, as in \cite{morse-wboundary}),
  then $\xi$ is a gradient-like vector field for the resulting Morse function
  $\wt{f}$ obtained by the Global Rearrangement Theorem.

  The statement of Global Rearrangement Theorem implies that there are no critical points in between $\wt{f}(p)$ and $\wt{f}(q)$.
  If there are any other critical points on the level set of $\wt{f}(p)$,
  we use the Elementary Rearrangement Theorem again to push them
  slightly below that level set.
  Likewise, any critical point on the level set of $\wt{f}(q)$ other than $q$ can be pushed slightly
  above that level set. As previously, $\xi$ is ensured to be gradient-like for $\wt{f}$.

  The resulting function satisfies the statement of Proposition~\ref{prop:no_points_in_between}, as desired.
\end{proof}

From now on, assume that such a rearrangement has been made.
In Section~\ref{sec:coordinate} we fixed the coordinate system on a neighborhood $U_2$ of $\gamma$ in which the gradient-like vector field $\xi$ for $f$
has the form described in \cref{lem:u2}. We choose now smaller neighborhood $\WW$ of $U_2$ containing $\gamma$: properties of $\WW$ will
be specified later.
As $x_1$ and $y$ play a special role in the proof of Theorem~\ref{thm:main}, we will use a slightly different notation for coordinates.
Namely, we will use the coordinates $(y,x,\bar u)$ (where $\bar u=(u_1,\dots,u_{n-2})$),
with $x$ playing the role of $x_1$ and $\bar u$ being the vector
$(x_2,\dots,x_n)$. That is to say, inside $U_2$ (in particular, inside $\WW$), the vector field $\xi$ has the form
\begin{equation}\label{eq:form_of_xi_u}
  (y(2x-1),w(x),-u_1,\dots,-u_{k-1},u_k,\dots,u_{n-2}),
\end{equation}

Set $a=\inf_{w\in \WW}f(w)$, $b=\sup_{w\in \WW}f(w)$. Upon possibly shrinking $\WW$,
by using Proposition~\ref{prop:no_points_in_between},
we may assume that the following condition is satisfied.
\begin{cond}\label{cond:ition}
  There are no critical points $q'$ of $f$ such that $f(q')\in[a,b]$ and $q'\neq p,q$.
\end{cond}

Possibly shrinking $\WW$ even further we may assume
that $\WW$ has the form $\WWII\times (-\delta,\delta)^{n-2}$ where $\WWII$ is a simply-connected open subset of $\R^2$
and $\delta>0$.
Recall that the function $w$ is positive on $(0,1)$ and negative away from $[0,1]$.
We choose $U\subset \WW$ as a neighborhood of $\gamma$ in such a way that the following holds:
\begin{cond}\label{cond:ition2}
  A trajectory of $\xi$ going through $U$ and $\WW$ never returns
  to $U$.
\end{cond}
Existence of such $U$ was proved in Proposition~\ref{prop:leaving-crit}.

First perturb $\xi$ by the formula.
\begin{equation}\label{eq:xic_def}
  \xi_c = \xi - c \eta(y, x, \bar u) \pd{}{x}
\end{equation}
where $c>0$ is a constant such that the $x$ component of $\xi_c$ is always negative on $\WW\cap\partial M$,
and $\eta$ is a suitable bump function supported in $\WW$.
For convenience, we will take $\eta$ of the form $\eta(y, x, \bar u) = \alpha(x) \beta(y) \delta(\bar u)$,
where $\alpha, \beta,\delta$
are  such that:
\begin{enumerate}[label=(E-\arabic*)]
  \item $\eta$ is equal to $1$ in a neighborhood of $\gamma$;
  \item the support of $\eta$ is contained within $U$;
  \item $\beta'(y) \ne 0$ for all $y$ such that $\beta(y)\neq 0,1$;\label{item:beta}
  \item $\alpha(x)=1$ for all $x\in[0,1]$;
  \item $\delta\equiv 1$ in a neighborhood of $(0,\dots,0)$.\label{item:equiv}
\end{enumerate}

We will now study the properties of $\xi_c$.
\begin{lemma}\label{lem:tangent}
  The vector field $\xi_c$ is tangent to $\partial M$.
\end{lemma}
\begin{proof}
  This follows from the fact that both $\xi$ and $\pd{}{x}$ are tangent to $\partial M$.
\end{proof}

\begin{proposition}\label{txi-crits}
  The critical points of $\xi_c$ away from $\WW$ coincide with the critical points of $\xi$. In $\WW$, $\xi_c$
  has a single critical point $z=(y_0,x_0,0,\dots,0)$, where $x_0 = \frac{1}{2}$ and $y_0$ is uniquely specified
  by the condition
  \begin{equation}\label{eq:new-crit}
    \beta(y_0)=\frac1cw(x_0).
  \end{equation}
\end{proposition}
We recall that $w(x)$ is the function specified by Proposition~\ref{prop:coordinates}; cf. \eqref{eq:form_of_xi_u}.
\begin{proof}
  Clearly, $\xi$ and $\xi_c$ coincide outside $U$; in particular, $\xi_c$ has no critical points inside $\WW\setminus U$.
  Note that $\xi_c$ has a non-vanishing $x$ coordinate on $\gamma$, hence neither $p$ nor $q$
  are critical points of $\xi_c$.

  Next, suppose $z=(y_0,x_0,\bar u_0)$ satisfies $\xi_c(z)=0$. The vanishing
  of the $\bar u$ components of $\xi_c(z)$ is equivalent to saying that $\bar u_0=(0,\dots,0)$.
  The vanishing of the $y$ coordinate of $\xi_c(z)$ implies that $y_0(2x_0-1)=0$. Now $y_0=0$ would mean that $z\in U\cap\partial M$, but
  then the $x$ component of $\xi_c(z)$ is non-zero. Hence, $x_0=\frac12$. This implies that
  \[
    \xi_c(z)=\left(0,w(x_0)-c\beta(y_0),0,\dots,0\right).
  \]
  Consequently, $y_0$ satisfies \eqref{eq:new-crit}. Since $w(x_0)>0$ (as per Proposition~\ref{prop:coordinates})
  and $\beta$ is strictly decreasing on the set where it takes values between $0$ and $1$ (by \ref{item:beta}),
  such $y_0$ is unique.
\end{proof}
We now study the critical point $z$ in greater detail.
\begin{lemma}\label{lem:eigenvalues}
  The critical point $z$ is hyperbolic. The linearization $D_{z}\xi_c$ has $k$ real negative eigenvalues and $n-k$
  real positive eigenvalues.
\end{lemma}
\begin{proof}
  The $\bar u$-coordinates account for $k-1$ real negative eigenvalues and $n-k-1$ real positive eigenvalues.
  We have $\delta\equiv 1$ in a neighborhood of $z$.
  We can now restrict our attention to the initial two coordinates of $\xi_c$:
  \begin{equation}\label{eq:xic_coor}
    \xi_2(y, x) = (y(2x-1), w(x)-c\alpha(x)\beta(y)).
  \end{equation}
  The derivative of $\xi_c$ in these directions is given by
  \[
    D_{y, x}\xi_2 =
    \begin{bmatrix}
      \pd{\xi_2}{y} & \pd{\xi_2}{x}
    \end{bmatrix}
    =
    \begin{bmatrix}
      2x - 1               & 2y                          \\
      -c\alpha(x)\beta'(y) & w'(x) - c\alpha'(x)\beta(y)
    \end{bmatrix}
  \]
  For $z=(y_0,x_0)$, since $\alpha \equiv 1$ on $[0, 1]$, we get
  $\alpha(x_0) = 1, \alpha'(x_0) = 0$ and the above simplifies to:
  \begin{equation}\label{eq:dxy}
    D_z\xi_2 =
    \begin{bmatrix}
      0             & 2y_0    \\
      -c\beta'(y_0) & w'(x_0)
    \end{bmatrix}
  \end{equation}
  We know that $\beta'(y_0)<0$ and $y_0>0$. That is, the matrix \eqref{eq:dxy} has a negative determinant, and so, regardless
  of the sign of $w'(x_0)$, it has two real eigenvalues, one positive and one negative.
\end{proof}
The vector field $\xi_c$ has a hyperbolic critical point at $z$. We need to make sure that $\xi_c$
has the form required by \cref{def:gl-vf} near $z$.
To this end, we simply
replace $\xi_c$ near $z$ by its linear part. The new vector field has the form \eqref{eq:morse_normal} near $z$.

\begin{proposition}\label{assure-normalform}
  For any neighborhood $U_{z}\subset U$ of $z$,
  there exists a vector field $\xip$, agreeing with $\xi_c$ away from $U_{z}$,
  such that, in some system of coordinates on $U_{z}$, $\xip$ has the form \eqref{eq:morse_normal}.
  Moreover, there exists a $C>0$ such that
  \begin{eqnarray*}
    \norm{\xip(z')-\xi_c(z')} &<& C\norm{z-z'} ^2 \\
    \norm{\xip(z')-\xil(z')} &<& C\norm{z-z'} ^2
  \end{eqnarray*}
  for any $z'\in U_z$.
\end{proposition}
\begin{proof}
  Let $\xil$ be the linear part of $\xi_c$ at $z$.
  Let $\tau$ be a bump function supported on $U_{z}$, equal to $1$ on a smaller neighborhood of $z$.
  Set
  \[
    \xip= \xi_c(1-\tau)+\xil\tau.
  \]
  Then, as long as $\tau\equiv 1$, we have $\xip=\xil$, which means that in the system of coordinates corresponding to the eigenvectors of $\xi_c(z)$,
  the vector field $\xip$ has the form \eqref{eq:morse_normal}, as desired.

  Note that the $\xip$, $\xi_c$, and $\xil$ have common linear terms at $z$. Therefore, the difference between $\xip$, $\xi_c$ and $\xil$
  at a point $z'$ is
  of order $\norm{z-z'}^2$ by the Taylor expansion formula. Given that, a standard argument allows us to find a $C>0$ satisfying the statement of the
  lemma.
\end{proof}

We will now aim to construct a new Morse function, whose gradient-like vector field will be $\xip$.
The main step towards this result is the following lemma:
\begin{lemma}\label{lem:key_lemma}
  Any trajectory of $\xip$ that enters $\WW$, either converges to $z$ or leaves $\WW$ in finite time.
\end{lemma}
\begin{proof}
  The proof is done in three steps. In Step~1,
  we provide a proof in two-dimensional case for the vector field $\xi_c$. In Step~2, we will show that this argument works
  also for $\xip$. In Step~3, we show that the higher-dimensional case can be reduced to the two-dimensional case.

  \pfstep{1}{The case of $\xi_c$ in two dimensions}{lem:key_lemma:step1}
  \begin{figure}
    \begin{tikzpicture}[scale=3]
      \draw[ultra thin,->] (0, 0) -- node[scale=0.8, below, pos=1] {$y$} (2.1, 0);
      \draw[ultra thin,->] (0, -0.5) -- node[scale=0.8, right, pos=1] {$x$} (0, 1.5);

      \coordinate (P) at (0, 0);
      \coordinate (Q) at (0, 1);

      \draw[ultra thick,red] (0, 0.5) -- node[below,scale=0.8, very near end] {$\Gamma_y$}(2, 0.5);
      \draw[ultra thick,red] (0, 1.2) -- (0, -0.2);

      \draw (P) node [above left, scale=0.8] {$p$};
      \fill (P) circle (0.02);
      \draw (Q) node [below left, scale=0.8] {$q$};
      \fill (Q) circle (0.02);

      \draw[blue!70!red,ultra thick] (1,1) --  node[near end, above, scale=0.8] {$\Gamma_x$} (2,1);
      \draw[blue!70!red, ultra thick] (1,0) -- (2,0);
      \draw[blue!60!green,ultra thick] (1,1) .. controls ++(-0.3, 0) and ++(-0.3, 0) ..  (1,0);
      \fill[blue] (1,1) circle (0.03);
      \fill[blue] (1,0) circle (0.03);


      \coordinate (A) at (1.3, 0.7);
      \draw[green!50!black, ->] (A) -- ++ (0.1,0);
      \draw[green!50!black, ->] (A) -- ++ (0,0.1);
      \fill[green!50!black] (A) circle (0.02);
      \draw (A) + (0.4, 0.1) node [below left, scale=0.8] {$\Omega_1$};

      \coordinate (B) at (0.6, 0.8);
      \draw[green!50!black, ->] (B) -- ++ (0.1,0);
      \draw[green!50!black, ->] (B) -- ++ (0,-0.1);
      \fill[green!50!black] (B) circle (0.02);
      \draw (B) + (-0.2, 0) node [below left, scale=0.8] {$\Omega_2$};

      \coordinate (C) at (1.3, 0.2);
      \draw[green!50!black, ->] (C) -- ++ (-0.1, 0);
      \draw[green!50!black, ->] (C) -- ++ (0, 0.1);
      \fill[green!50!black] (C) circle (0.02);
      \draw (C) + (0.4, 0.1) node [below left, scale=0.8] {$\Omega_4$};

      \coordinate (D) at (0.6, 0.3);
      \draw[green!50!black, ->] (D) -- ++ (-0.1,0);
      \draw[green!50!black, ->] (D) -- ++ (0,-0.1);
      \fill[green!50!black] (D) circle (0.02);
      \draw (D) + (-0.2, 0) node [below left, scale=0.8] {$\Omega_3$};

      \coordinate (Z) at (0.77, 0.5);
      \fill (Z) circle (0.03);
      \draw (Z) node [below right, scale=0.8] {$z$};

      \draw[gray, dashed] (0, 1.2) -- node[above] {$\WW$} (2, 1.2) -- (2, -0.2) -- (0, -0.2);
      \draw[gray,thin] (1.5,1) -- ++(0.3,0.3) -- node [above,near end, scale=0.7] {$\Gamma^0_x$} ++(0.5,0);
      \draw[gray,thin] (1.5,0) -- ++(0.3,-0.3) -- node [above,near end, scale=0.7] {$\Gamma^0_x$} ++(0.5,0);
      \draw[gray,thin] (0.79,0.6) -- node[above, very near end, scale=0.7] {$\Gamma^\kappa_x$} ++(-1.7,0);
    \end{tikzpicture}
    \caption{Proof of \Cref{lem:key_lemma}. Step 1. The arrows indicate possible directions of the vector field $\xi_c$. For example, at any point of $\Omega_1$, $\xi_c$ belongs to $\R_{\ge0}\times\R_{\ge0}$.}\label{fig:key_lemma}
  \end{figure}

  We proceed by analyzing the phase portrait. Recall that in the two-dimensional case, $\xi_c$ is given
  by the formula \eqref{eq:xic_coor}. We denote the coordinates of $\xi_c$ by
  \begin{align*}
    \xi_{c,x} & = w(x)-c\alpha(x)\beta(y) \\
    \xi_{c,y} & = y(2x-1)
  \end{align*}
  Let
  \[
    \Gamma_x = \xi_{c,x}^{-1}(0), ~~~ \Gamma_y = \xi_{c,y}^{-1}(0).
  \]
  The latter is easily seen as $\Gamma_y = \br{x = \frac{1}{2}} \cup \br{y = 0}$.

  In order to obtain a more explicit description of $\Gamma_x$,
  note that $\beta^{-1}(r)$ is well-defined for any $r \in (0,1)$.
  Set \[
    \kappa(x) = \beta^{-1}\left(\frac{w(x)}{c\alpha(x)}\right).
  \]
  Since $w$ is negative away from $[0,1]$, all points of $\Gamma_x$ satisfy $x \in [0, 1]$ and $\alpha(x) = 1$.
  Then $(\kappa(x),x)_{x \in (0, 1)}$ parametrizes the subset
  \[
    \Gamma_x^\kappa = \br{ (y, x) \in \Gamma_x \mid \beta(y) \ne 0, \beta(y) \ne 1}.
  \]
  However, if $\beta(y) = 1$, then \[ \xi_{c,x}(y, x)=w(x)-c < 0 \] by construction, so $\Gamma_x\cap \beta^{-1}(1)$ is empty.
  Similarly, if $(y, x) \in \Gamma_x$ and $\beta(y) = 0$, then \[ \xi_{c,x}(y,x)=w(x) = 0, \] that is, $x \in \br{0, 1}$. Therefore,
  $\Gamma_x=\Gamma_x^\kappa \cup \Gamma_x^0$, where
  \begin{align*}
    \Gamma_x^\kappa & = \br{ (\kappa(x),x) \mid x \in (0, 1)}     \\
    \Gamma_x^0      & = \br {(y,x) \mid x\in\{0,1\},\ \beta(y)=0}
  \end{align*}
  Notice that $\beta'(y) = 0$ only if $\beta(y) \in {0, 1}$. Now $\kappa(x)\in(0,1)$ for $x\in(0,1)$. Therefore,
  applying the inverse function theorem lets us conclude that
  \begin{equation}\label{eq:kappa-der}
    \kappa'(x) = \frac{1}{\beta'(\kappa(x))} \textrm{ is finite for } x \in (0, 1).
  \end{equation}
  In the light of \eqref{eq:kappa-der}, the vector $\frac{\partial}{\partial x}$ always crosses $\Gamma_x^\kappa$ transversely.
  Similarly, $\frac{\partial}{\partial x}$ always crosses $\Gamma_y$ transversely.
  Therefore, it is clear from the phase portrait that each trajectory of $\xi_c$ can cross $\Gamma_x^\kappa$ or $\Gamma_y$ only once.

  The level sets $\Gamma_x$ and $\Gamma_y$ divide $\WW$ into four regions
  $\Omega_1,\dots,\Omega_4$, as in \Cref{fig:key_lemma}. Let $\theta$ be the forward trajectory of $\xi_c$ starting at $p_s$.
  Notice that $\theta$ cannot cross $\Gamma_x^0$ because $\xi_c$ is tangent to $\Gamma_x^0$. Now:
  \begin{itemize}
    \item if $p_s \in \Omega_1$, then $\xi_{c,y}$ remains positive and separated from $0$ along $\theta$, so $\theta$ leaves $\WW$ to the right;
    \item if $p_s \in \Omega_3$, then $\xi_{c,x}$ remains negative and separated from $0$ along $\theta$, so $\theta$ always leaves $\WW$ to the bottom;
    \item if $p_s \in \Omega_2$ or $p_s \in \Omega_4$, then $\theta$ either
          \begin{itemize}
            \item hits the critical point $z$;
            \item leads to one of $\Omega_1$ or $\Omega_3$, and, by the previous considerations, eventually leaves $\WW$.
          \end{itemize}
  \end{itemize}
  In other words, any trajectory of $\xi_c$  entering $\WW$ either hits $z$ or leaves $\WW$ in finite time.

  \pfstep{2}{The case of $\xip$ in two dimensions}{lem:key_lemma:step2}

  We begin with the following observation. For any $U_0\subset\WW$ containing $z$ and sufficiently small,
  there exists another $U_1 \subset U_0$ being an open neighborhood of $z$,
  such that if the trajectory of $\xi_c$ starts in $U_1$ and then leaves $U_0$,
  then does not return to $U_1$, unless it leaves $\WW$.
  To see this, we use an argument similar to the one used in the proof of Proposition~\ref{prop:leaving-crit}
  given in \cite{milnor-hcobordism}.

  Namely, if for all $U_1$ the trajectory of $\xi_c$ leaves $U_0$ and subsequently returns to $U_1$,
  upon passing to a limit, we would construct a trajectory both starting and terminating at $z$.
  It is clear from the phase portrait, that every trajectory starting at $z$ leads either to $\Omega_1$ or to $\Omega_3$.
  Similarly, every trajectory entering $z$, comes from either $\Omega_2$ or $\Omega_4$.
  We already argued that the trajectory in $\Omega_1$
  and $\Omega_3$ cannot enter $\Omega_2$ or $\Omega_4$ without leaving $\WW$, hence no trajectory starts and terminates at $z$
  without leaving $\WW$.

  Before choosing an appropriate neighborhood $U_0$, we need to provide some estimates.
  Let $v_+$ and $v_-$ be length one eigenvectors of the matrix $D_{z}\xi_c$, where $v_+$ corresponds to
  the positive eigenvalue, while $v_-$ corresponds to the negative eigenvalue.
  Let $(s,t)$ be local coordinates near $z$ such that $v_+$ and $v_-$ correspond to $(1,0)$ and $(0,1)$ respectively.
  In other words, in these coordinates, we have:
  \begin{align*}
    \xil(s,t)  & = (c_+ s,c_- t)                              \\
    \xi_c(s,t) & = (c_+ s,c_- t)+O\left(\norm{(s,t)}^2\right)
  \end{align*}
  for the eigenvalues $c_+,c_-$ of $D_z\xi_c$, satisfying $c_+>0>c_-$.

  Consider the function
  \[ 
    g(s,t)=\frac{1}{2}\left(s^2-t^2\right).
  \] 
  With this definition, we have \[\partial_{\xil}g=c_+s^2-c_-t^2\ge 0\] with equality only at $(0,0)$.
  \Cref{assure-normalform} implies that the difference between $\xil$ and $\xi_c$ is quadratic in $s$ and $t$.
  Therefore, the difference between $\partial_{\xil}g$ and $\partial_{\xi_c}g$ is cubic in $s$ and $t$.
  That is to say,
  \[
    \partial_{\xi_c}g=c_+s^2-c_-t^2+O\left(\norm{(s,t)}^3\right).
  \]
  In particular, there exists a  neighborhood $U_0$ of $z$, such that $\partial_{\xi_c}g>0$ everywhere except at $z$.
  Find a smaller $U_1 \subset U_0$ such that any trajectory of $\xi_c$ entering $U_1$ leaving $U_0$ does not return
  to $U_1$.

  Let $\tau$ be a cut-off function supported in $U_1$ and
  let $\xip$ be the vector field constructed in~\Cref{assure-normalform} above using this function $\tau$.
  As $\partial_{\xil}g,\partial_{\xi_c}g>0$ on $U_1\setminus\{z\}$, the same inequality holds for a convex combination thereof.
  Let $\theta$ be a trajectory of $\xip$.
  \begin{itemize}
    \item If $\theta$ stays forever in $U_0$, then $\partial_{\xip}g\ge 0$ implies that $\theta$ must hit $z$;
    \item if $\theta$ does not enter $U_1$, then it is actually a trajectory of $\xi_c$ that does not hit $z$, so it must leave $\WW$;
    \item finally, suppose $\theta$ starts in $U_1$, does not hit $z$ and eventually leaves $U_0$.
          We know that this trajectory cannot return to $U_1$.
          Then, as soon as $\theta$ leaves $U_1$, it remains a trajectory of $\xi_c$. In particular,
          it is a trajectory of $\xi_c$ not hitting $z$. Therefore, $\theta$ must leave $\WW$.
  \end{itemize}
  These case conclude the proof of Step~2.

  \pfstep{3}{The general case}{lem:key_lemma:step3}
  We assume that $\WW$ has a product structure $\WW=\WWII\times I^{n-2}$, where $I=(-\sigma,\sigma)$ for some $\sigma>0$ and $\WWII$
  is an open contractible subset of $\R^2$. Consider the projection $\Pi\colon \WW\to \WWII$ given by
  $\Pi(y,x,\bar u)=(y,x)$. As passing from $\xi$ to $\xi_c$ affects only the $x$ coordinate (compare \eqref{eq:xic_def}),
  and changing
  $\xi_c$ to $\xip$ affects only the first two coordinates,
  we conclude that $\xip$ has the form
  \[
    \xip = (\alpha_1(y,x,\bar u),\alpha_2(y,x,\bar u),-u_1,\dots,-u_{k-1},u_k,\dots,u_{n-2}),
  \]
  where $\alpha_1$ and $\alpha_2$ are smooth.
  Suppose $\theta$ is a trajectory of $\xip$ staying in $\WW$ forever.
  The $\bar u$ components of $\xip$ are $(-u_1,\dots,-u_{k-1},u_k,\dots,u_{n-2})$, that is,
  if $\theta$ stays forever in $\WW$, then the $u_k,\dots,u_{n-2}$ coordinates have to be zero on $\theta$ (otherwise,
  they eventually become very large). Conversely, the $u_1,\dots,u_{k-1}$, are decreasing along $\theta$, and so, all the $\bar u$-coordinates
  converge to $(0,\dots,0)$. By~\ref{item:equiv}, it follows that the trajectory
  of $\theta$ eventually is contained in the set $\br{\delta(\bar u)=1}$. That is to say, $\Pi(\theta)$ is a trajectory
  of the vector field $\xip_2$ in $\WWII$, where
  \[
    \xip_2 = (\alpha_1(y,x,0,\dots,0),\alpha_2(y,x,0,\dots,0))
  \]
  is easily to seen to be
  precisely the vector field $\xip$ from the two-dimensional discussion in the previous step.
  Any trajectory of $\xip_2$ either converges to the critical point $(y_0, x_0)$
  or leaves $\WWII$. Therefore, if $\theta$ stays forever in $\WW$, $\Pi(\theta)$ has to converge to $(y_0,x_0)$.
  As we have already mentioned, in that case, the $\bar u$-coordinates of $\theta$ converge to $0$. Consequently, if $\theta$ stays forever in $\WW$, then it terminates at $z$.
\end{proof}
\begin{remark}
  What we proved in Lemma~\ref{lem:key_lemma} concerns the forward behavior of a trajectory. However, the backward result is also true:
  if a trajectory stays forever in the past in $\WW$, it has to start at $z$. The proof is completely analogous.
\end{remark}
As a corollary, we prove the following result.
\begin{corollary}
  Suppose $a,b$ are such that $\WW\subset f^{-1}[a,b]$ and the only critical points of $f$ in $f^{-1}[a,b]$ are $p$ and $q$.
  If $\theta$ is a trajectory of $\xip$, then:
  \begin{itemize}
    \item Either it exits $f^{-1}[a,b]$ through $f^{-1}(b)$, or it terminates at $z$.
    \item Either it enters $f^{-1}[a,b]$ through $f^{-1}(a)$, or it starts at $z$;
  \end{itemize}
\end{corollary}
\begin{proof}
  Recall that the set $U$ was defined as a subset of $W$ on which all the changes are to be made, that is $\xi=\xi_c=\xi'$ away
  from $U$ (see Condition~\ref{cond:ition2}). That is to say, if $\theta$ does not intersect $U$, it is a trajectory of $\xi$. Any trajectory of $\xi$ that does not hit $U$, flows
  from $f^{-1}(a)$ to $f^{-1}(b)$.

  Suppose $\theta$ enters $U$. By Lemma~\ref{lem:key_lemma}, either it hits $z$, or it leaves $\WW$.
  If it hits $z$, we are done;
  if it leaves $W$, then by Proposition~\ref{prop:leaving-crit}, it does never return to $U$, but then,
  as soon as $\theta$ leaves $\WW$, it actually becomes a trajectory of $\xi$, and so it must terminate at a critical point.
  As there are no critical points of $f$ in $f^{-1}[a,b]\setminus W$, $\theta$ must hit a critical point with a critical value not in $[a,b]$.
  Since $f$ increases along $\theta$, we conclude that $\theta$ terminates above the level set $f^{-1}(b)$.
  This proves the first part of the corollary. The proof of the other is analogous.
\end{proof}

As a final stage, we construct a Morse function $g$ whose gradient-like vector field is $\xip$.
The construction follows from Vector Field Integration Lemma \cite{BorodzikPowell}, which in turn generalizes
\cite[Assertion 5, page 54]{milnor-hcobordism}. Nevertheless, the proof in \cite{BorodzikPowell} has a small technical
flaw (namely, the function that is constructed is not necessarily continuous unless the vector field is properly rescaled),
therefore we give an independent construction. We note that in \cite{BorodzikPowellTeichner}, there will
be given a more general statement of the Vector Field Integration Lemma.
The following result concludes the proof of Theorem~\ref{thm:main}.
\begin{proposition}\label{prop:integrate}
  There exists a Morse function $g$ whose gradient-like vector field is $\xip$,
  such that $g=f$ away from $f^{-1}(a,b)$.
\end{proposition}
\begin{proof}
  The high-level idea is to define $g$ by an explicit formula in a neighborhood of $z$
  and by interpolation elsewhere.

  Set $c=\frac12(a+b)$. Inside $U_{z}$ (see Proposition~\ref{assure-normalform}) define a function
  \[g_0(s,t,u_1,\dots,u_{n-2})=c+s^2-t^2-u_1^2-\dots-u_{k-1}^2+u_k^2+\dots+u_{n-2}^2.\]
  Write $\alpha^2:=s^2+u_k^2+\dots+u_{n-2}^2$, $\beta^2:=t^2+u_1^2+\dots+u_{k-1}^2$.
  For $\rho>\varepsilon>0$ define (cf. \cite[Section 2.4]{morse-wboundary}) the subset of $U_{z}$:
  \[H_{\rho,\varepsilon}:=\{-\alpha^2+\beta^2\in [-\varepsilon^2,\varepsilon^2],\ \ \alpha^2\beta^2\leqslant (\rho^4-\varepsilon^4)/4\}.\]
  For sufficiently small $\rho$, the set $H_{\rho,\varepsilon}$ is compact (which is a formal way of saying that it is contained
  in the interior of $U_{z}$).
  Let us now define the following parts of the boundary of $H_{\rho,\varepsilon}$; see Figure~\ref{fig:u_rho}.
  \begin{equation}\label{eq:bnd}
    \begin{split}
      \Xin&=\partial  H_{\rho,\varepsilon}\cap \{-\alpha^2+\beta^2=\varepsilon^2\}\subset g_0^{-1}(c-\varepsilon^2)\\
      \Xout&=\partial H_{\rho,\varepsilon}\cap \{-\alpha^2+\beta^2=-\varepsilon^2\}\subset g_0^{-1}(c+\varepsilon^2)\\
      \Xtan&=\partial H_{\rho,\varepsilon}\cap \{\alpha^2\beta^2=(\rho^4-\varepsilon^4)/4\}.\\
    \end{split}
  \end{equation}

  \begin{figure}
    \begin{tikzpicture}
      \fill[blue,opacity=0.2] (2.2,0) -- ++(0.3,0.3) .. controls ++ (-1.5,0) and ++ (0,-1.5) .. (0.3,2.5) -- (0,2.2) -- (0,0) -- (2.2,0);
      \draw[very thick, green] (2.2,0) -- ++ (0.3,0.3);
      \draw[very thick, blue] (0,2.2) -- ++ (0.3,0.3);
      \draw[very thick, purple] (2.5,0.3) .. controls ++ (-1.5,0) and ++(0,-1.5) .. (0.3,2.5);
      \draw[->] (-0.5,0) -- (3,0);
      \draw[->] (0,-0.5) -- (0,3);
      \draw (3,-0.2) node [scale=0.8] {$\alpha^2$};
      \draw (-0.2,3) node [scale=0.8] {$\beta^2$};
      \draw[green!50!black] (2.4,0.1) -- ++(0.5,0.5) -- node [above,scale=0.8,midway] {$\Xin$} ++(1,0);
      \draw[blue!50!black] (0.1,2.4) -- node [scale=0.8,above, near end] {$\Xout$} ++ (-1.3,0);
      \draw[purple!50!black] (1.1,1.1) node [scale=0.8] {$\Xtan$};
    \end{tikzpicture}
    \caption{A schematic presentation of $H_{\rho,\varepsilon}$.}\label{fig:u_rho}
  \end{figure}
  \newcommand{\Heps}[1]{H_{\varepsilon_{#1}}}
  It is clear that $\xip$ is tangent to $\Xtan$ (cf. \cite[Lemma 2.31]{morse-wboundary}).
  Choose sufficiently small values $\rho > \varepsilon_1>\varepsilon_2 > 0$.
  For brevity, we will write $\Heps{i} := H_{\rho, \varepsilon_i}$ and the corresponding subsets
  of $\partial \Heps{i}$ will be denoted by $\Xin^i,\Xout^i,\Xtan^i$, as in  Figure~\ref{fig:flow}.

  \begin{figure}
    \begin{tikzpicture}
      \fill[yellow,opacity=0.1] (-3,-2) -- (3,-2) -- (3,2) -- (-3,2) -- (-3,-2);
      \draw[thick] (-3,-2) -- (3,-2); \draw (3.5,-2) node [scale=0.8] {$f^{-1}(a)$};
      \draw[thick] (-3,2) -- (3,2); \draw (3.5,2) node [scale=0.8] {$f^{-1}(b)$};
      \fill[green,opacity=0.1] (-1.2,-0.6) -- (1.2,-0.6) .. controls ++(-0.3,0.3) and ++ (-0.3,-0.3) .. (1.2,0.6) -- (-1.2,0.6) .. controls ++(0.3,-0.3) and ++(0.3,0.3) .. (-1.2,-0.6);
      \fill[blue,opacity=0.3] (-0.5,-0.6) -- (0.5,-0.6) .. controls ++(-0.3,0.3) and ++ (-0.3,-0.3) .. (0.5,0.6) -- (-0.5,0.6) .. controls ++(0.3,-0.3) and ++(0.3,0.3) .. (-0.5,-0.6);
      \draw[very thin, dashed] (1.2,-0.6) .. controls ++ (-0.3,0.3) and ++(-0.3,-0.3) .. (1.2,0.6);
      \draw[very thin, dashed] (-1.2,-0.6) .. controls ++ (0.3,0.3) and ++(0.3,-0.3) .. (-1.2,0.6);
      \draw (-1.2,-0.6) -- (1.2,-0.6); \draw (1.6,-0.6) node [scale=0.8] {$\Xin^1$};
      \draw (-1.2,0.6) -- (1.2,0.6); \draw (1.6,0.6) node [scale=0.8] {$\Xout^1$};
      \fill[black] (0,0) circle(0.08);
      \draw(0,0.3) node [scale=0.8] {$z$};
      \draw[thin] (0,-0.4) -- ++(-0.4,-0.4) -- node[below, scale=0.8] {$\Heps{2}$} ++(-1,0);
      \draw[thin] (-0.9,0.4) -- ++ (-0.4,0.6) -- node[above, scale=0.8] {$\Heps{1}$} ++(-1,0);

    \end{tikzpicture}
    \caption{A schematic picture of various sets in Proposition~\ref{prop:integrate}. We assume that level sets of $f$ are horizontal, and the flow of $\xi'$ is directed upwards. The vector field $\xi'$ is tangent to the `vertical' boundaries of $H_{\varepsilon_1}$ and $H_{\varepsilon_2}$ and transverse to $\Xin^1$ and $\Xout^1$. Note that for simplicity, we ignore the change of topology of the level sets of $f$ while passing critical points.}\label{fig:flow}
  \end{figure}

  Rescale the vector field $\xip$ by a positive factor
  in such a way that:
  \begin{itemize}
    \item if the flow of $\xip$ does not hit $\Heps{1}$, then it takes precisely the time $b-a$ to get
          from $f^{-1}(a)$ to $f^{-1}(b)$;
    \item if the flow of $\xip$ starts from $f^{-1}(a)$ and hits $\Xin^1$, then it takes precisely time equal $c-\varepsilon-a$;
    \item if the flow of $\xip$ ends at $f^{-1}(b)$ and hits $\Xout^1$ in the past, then it takes precisely time equal
          $b-c-\varepsilon$;
    \item The time to reach $\Xout^1 \setminus \Xout^2$ from $\Xin^1 \setminus \Xin^2$ is equal to $2\varepsilon$.
  \end{itemize}

  \newcommand{\garg}{x}
  Choose a point $\garg \in f^{-1}(a,b)$ and let $\gamma_\garg$ be the trajectory of $\xip$ through $\garg$.
  We now define the function $g_1$ away from $\Heps{2}$ as follows:
  \begin{itemize}
    \item Suppose $\gamma_\garg$ travels from $f^{-1}(a)$ to $f^{-1}(b)$ without hitting $\Heps{2}$.
          Assume that $\gamma_\garg(0)\in f^{-1}(a)$, and so $\gamma_\garg(b-a)\in f^{-1}(b)$.
          We set $g_1(\garg)=a+t_\garg$, where $t_\garg$ is defined by $\gamma_\garg(t_\garg)=\garg$.
    \item Suppose $\gamma_\garg$ travels from $f^{-1}(a)$, passes through $\garg$, and then hits $\Xin^2$.
          Assume that $\gamma_\garg(0)\in f^{-1}(a)$.
          We set $g_1(\garg)=a+t_\garg$, where $t_\garg$ is defined by $\gamma_\garg(t_\garg)=\garg$.
    \item Suppose $\gamma_\garg$ travels from $\Xout^2$, passes through $\garg$, and then hits $f^{-1}(b)$.
          We assume $\gamma_\garg(0)\in \Xout^2$.
          We set $g_1(\garg)=c+\varepsilon+t_\garg$, where $t_\garg$ is defined by  $\gamma_\garg(t_\garg)=\garg$.
  \end{itemize}
  With this definition, we have defined $g_1(\garg)$ everywhere in $f^{-1}[a,b]$ except for the interior of $\Heps{2}$.
  For future reference, observe that the way the vector field $\xip$ leads to the following claim:
  \begin{lemma}\label{lem:for_future}
    The function $g_1$ is equal to $c-\varepsilon$ on the whole of $\Xin^1$, and to $c+\varepsilon$ on the whole of $\Xout^1$.
    Moreover, $f$ and $g_1$ coincide on $f^{-1}(b)$.
  \end{lemma}

  Ideally, we would like to set $g$ to be equal to $g_0$ on $\Heps{2}$ and to $g_1$ away from $\Heps{2}$.
  However, this could potentially introduce a discontinuity at the boundary.
  To avoid this, we choose a smooth cut-off function
  $\phi\colon \Xin^1\to[0,1]$, equal to $1$ on $\Xin^2$ and supported on a compact subset of $\Int \Xin^1$.
  Extend $\phi$ to the whole of $\Heps{1}$ demanding that $\phi$ be invariant under the flow of $\xip$.
  We set
  \begin{equation}\label{eq:def_of_g}
    g(w)=
    \begin{cases}
      g_1(w)                               & w\notin \Heps{1}, \\
      \left(\phi g_0+(1-\phi)g_1\right)(w) & w\in \Heps{1}.
    \end{cases}
  \end{equation}
  In particular $g = g_1$ outside $\Heps{1}$.
  We only need to check continuity at $\partial \Heps{1}$.
  By Lemma~\ref{lem:for_future}, $g$ is continuous at $\Xin^1$ and $\Xout^1$,
  the vanishing of $\phi$ on $\Xtan^1$ implies that $g$ is continuous at $\Xtan^1$, and so,
  $g$ is continuous everywhere.
  Moreover, after a little technical amendment, that is, after rescaling
  $\xip$, $g$ can be assumed smooth, cf. \cite[Proof of Proposition 2.35]{morse-wboundary}.
  By construction, $g=g_0$ in $\Heps{2}$,
  in particular, it is Morse.

  We claim that $\xip$ is gradient-like for $g$ constructed in such a way.
  Since $g$ and $\xi'$ admit the Morse normal form near $z$ (see \eqref{eq:morse_normal}),
  we only need to check that $\partial_{\xip}g \ge 0$ on $f^{-1}[a,b]$, with the sole equality at $z$.
  It is evident on $\Heps{2}$, as $g=g_0$.
  Note that, by construction, $\partial_{\xip}g_1\equiv 1$ on $f^{-1}[a,b]\setminus \Heps{2}$.
  In particular, $\partial_{\xip}g > 0$ away from $\Heps{1}$. Suppose $w\in \Heps{1}\setminus \Heps{2}$.
  The function $\phi$ was to be $\xip$-invariant, hence $\partial_{\xip}\phi=0$. By \eqref{eq:def_of_g}:
  \[
    \partial_{\xip}g = \partial_{\xip}(\phi g_0+(1-\phi)g_1) = \phi\partial_{\xip}g_0+(1-\phi)\partial_{\xip}g_1.
  \]
  As $\partial_{\xip}g_0(\garg)>0$ on $\Heps{1}\setminus \Heps{2}$, and $\partial_{\xip}g_1\equiv 1$,
  we conclude that $\partial_{\xip}g>0$ on $\Heps{1}\setminus \Heps{2}$.

  In this way, we finish the proof of Proposition~\ref{prop:integrate}, which was the last step in the proof of Theorem~\ref{thm:main}.
\end{proof}

\bibliographystyle{amsalpha}
\def\MR#1{}
\bibliography{morse}

\end{document}